\documentclass[12 pt]{amsart}
\usepackage{amsmath}
\usepackage{enumerate}
\usepackage{mathtools}
\usepackage{dashrule}
\usepackage{amssymb}
\usepackage{bm}
\usepackage{amssymb}
\usepackage{tabu}
\usepackage{tikz}
\usepackage{hyperref}
\usepackage[T1]{fontenc}
\usepackage{ucs}
\usepackage[utf8]{inputenc}
\usepackage{lmodern}
\usepackage{amssymb}
\usepackage[utf8]{inputenc}
\usepackage[T1]{fontenc}
\usepackage{ucs}

\usepackage{fancyhdr} 
\theoremstyle{theorem}
\newtheorem{theorem}{Theorem}
\newtheorem{lemma}{Lemma}

\theoremstyle{definition}

\numberwithin{equation}{section}

\begin{document}


\title{Sylvester's theorem and the Non-integrality of a certain Binomial Sum}


 \author{Daniel L\'{o}pez-Aguayo}
 \address{Centro de Ciencias Matem\'{a}ticas
UNAM,Morelia,Mexico}
\email{dlopez@matmor.unam.mx}

 \author{Florian Luca}
 \address{School of Mathematics \\
 University of the Witwatersrand \\
 Private Bag 3 Wits 2050 \\
 Johannesburg, South Africa}
\email{Florian.Luca@wits.ac.za}

\maketitle

\begin{abstract}
In this note, we show that 
$$
S(n,r):=\sum_{k=0}^{n} \frac{k}{k+r} \binom{n}{k}
$$
is not an integer for any positive integer $n$ and $r\in \{1,2,3,4,5,6\}$ and for $n\le r-1$. This gives a partial answer to a conjecture of \cite{3}.
\end{abstract}

\section{Introduction}

Marcel Chiri\unichar{539}\unichar{259} \cite{1} asked to show that
\begin{equation} 
\sum_{k=0}^{n} \frac{k}{k+1}\binom{n}{k}\not\in {\mathbb Z}
\end{equation}
for any integer $n \ge 1$. The first author \cite{3} proved that
$$
\sum_{k=0}^{n} \frac{k}{k+r}\binom{n}{k}
$$
is not an integer for positive integers $n$ and $r\in \{2,3,4\}$ and asked if the above sum is ever an integer for some positive integers $n$ and $r$. Plainly, since
$$
\sum_{k=0}^n \binom{n}{k}=2^n
$$ 
is an integer, the question is equivalent to whether
\begin{equation}
\label{eq:Snr}
S(n,r):=\sum_{k=0}^n \frac{r}{k+r}\binom{n}{k}
\end{equation}
is ever an integer for positive integers $n$ and $r$. For $n=1$, we have $S(n,r)=1+r/(r+1)\in (1,2)$ is not an integer, so we may assume that $n\ge 2$.
Trying out small values of $r$ we find the formulas:
\begin{eqnarray}
\label{eq:1to5}
S(n,1) & = & \frac{2^{n+1}-1}{n+1};\nonumber\\
S(n,2) & = & (-2) \left(\frac{2^{n+1}-1}{n+1}\right)+2\left(\frac{2^{n+2}-1}{n+2}\right);\nonumber\\
S(n,3) & = & 3\left(\frac{2^{n+1}-1}{n+1}\right)-6\left(\frac{2^{n+2}-1}{n+2}\right)+3\left(\frac{2^{n+3}-1}{n+3}\right);\nonumber\\
S(n,4) & = & (-4)\left(\frac{2^{n+1}-1}{n+1}\right)+12 \left(\frac{2^{n+2}-1}{n+2}\right)-12 \left(\frac{2^{n+3}-1}{n+3}\right)\nonumber\\
& + & 4\left(\frac{2^{n+4}-1}{n+4}\right);\\
S(n,5) & = & 5\left(\frac{2^{n+1}-1}{n+1}\right)-20\left(\frac{2^{n+2}-1}{n+2}\right)+30\left(\frac{2^{n+3}-1}{n+3}\right)\nonumber\\
& - & 20\left(\frac{2^{n+4}-1}{n+4}\right)+5\left(\frac{2^{n+5}-1}{n+5}\right);\nonumber\\
S(n,6) & = & (-6) \left(\frac{2^{n+1}-1}{n+1}\right)+30 \left(\frac{2^{n+2}-1}{n+2}\right)-60\left(\frac{2^{n+3}-1}{n+3}\right)\nonumber\\
& + & 60\left(\frac{2^{n+4}-1}{n+4}\right)-30\left(\frac{2^{n+5}-1}{n+5}\right)+6\left(\frac{2^{n+6}-1}{n+6}\right).\nonumber
\end{eqnarray}
At this point we recall the well-known fact that $n$ never divides $2^n-1$ for any $n\ge 2$ (see, for example, problem A14 in \cite{PEN}). 

In particular, 
$(2^{n+1}-1)/(n+1)$ is not an integer which by the first relation \eqref{eq:1to5} deals with the case $r=1$. 

For $r=2$, one of $n+1$ and $n+2$ is odd. 
We assume that $n+1$ is odd, since the case when $n+2$ is odd is similar.
Then, $2(2^{n+1}-1)/(n+1)$ is a rational number, which in its simplest form, has an odd prime $p$ in its denominator. Since $n+1$ and $n+2$ are coprime, we get that $p$ does not divide $n+2$, so $p$ divides the denominator of $S(n,2)$. Hence,  $S(n,2)$ is not an integer. 

For $r=3$, suppose first that $n+1$ is odd. Then so is $n+3$ and one of $n+1,~n+3$ is not a multiple of $3$. Assume $n+1$ is not a multiple of $3$, and the case when $n+3$ is not a multiple of $3$ can be dealt with similarly. Then $3(2^{n+1}-1)/(n+1)$ is a rational number, which in its simplest form, has a prime factor $p\ge 5$ in its denominator. Clearly, $p$ does not divide either one of $n+2,~n+3$, so $p$ divides the denominator of $S(n,3)$. Hence, $S(n,3)$ is not an integer. Assume now that $n+1$ is even. In this case, one of $n+1,~n+3$ is a multiple of $4$, and the other is congruent to $2\pmod 4$, and plainly $n+2$ is odd. The third formula \eqref{eq:1to5} now shows easily that $S(n,3)$ is not a $2$-adic integer in this case. In fact, its denominator as a rational number is a multiple of $4$. This takes care of the case $r=3$.

For $r=4$, either $n+1$ or $n+4$ is odd. We assume that $n+1$ is odd since the case when $n+4$ is odd can be dealt with similarly. Then $n+1$ and $n+3$ are both odd and at most one of them is a multiple of $3$. Thus, there exists $i\in \{1,3\}$ such that $n+i$ is coprime to $6$. Then $c_i(2^{n+i}-1)/(n+i)$ 
is a rational number whose denominator is divisible by a prime $p\ge 5$. Here, $c_i=4$ if $i=1$ and $c_i=12$ if $i=3$. This prime $p$ cannot divide 
$n+j$ for any $j\ne i$, $j\in \{1,2,3,4\}$, therefore $p$ divides the denominator of $S(n,4)$.

For $r=5$, consider first the case when $n+1$ is odd. Then $n+1,~n+3,~n+5$ are all odd. Of these three numbers, at most one is a multiple of $3$ and at most one is a multiple of $5$. Hence, there is $i\in \{1,3,5\}$ such that $n+i$ is coprime to $30$. Then $c_i(2^{n+i}-1)/(n+i)$ is a rational number 
whose denominator is a multiple of a prime $p\ge 7$. Here, $c_i=5,30,5$, for $i=1,3,5$, respectively. The prime $p$ cannot divide $n+j$ for any $j\ne i$, $j\in \{1,2,3,4,5\}$, so $S(n,5)$ is not an integer. Assume now that $n+1$ is even. If $n+1\equiv 2\pmod 4$, then $n+3\equiv 0\pmod 4$ and $n+5\equiv 2\pmod 4$. 
Hence, 
$$
5\left(\frac{2^{n+1}-1}{n+1}\right)+30\left(\frac{2^{n+3}-1}{n+3}\right)+5\left(\frac{2^{n+5}-1}{n+5}\right)
$$
is a rational number which, in its simplest form, has an even denominator. Since $n+2,~n+4$ are odd, it follows that $S(n,5)$ is a rational number with an even denominator. 
Finally, when $n+1\equiv 0\pmod 4$, then $n+3\equiv 2\pmod 4$ and $n+5\equiv 0\pmod 4$. Since $n+1,~n+5$ are both multiples of $4$ whose difference is $4$, it follows that one of them is congruent to $4\pmod 8$ and the other is a multiple of $8$. It now follows that the denominator of $S(n,5)$ is even, and in fact, 
is a multiple of $8$. Hence, $S(n,5)$ is not an integer either.

For $r=6$, one of $n+1$ to $n+6$ is odd. We consider only  the case when $n+1$ is odd since the case when  $n+6$ is odd is similar. Then $n+1,~n+3,~n+5$ are all odd and at most one of them is a multiple of $3$ and at most one of them is a multiple of $5$. Hence, there is $i\in \{1,3,5\}$ such that $n+i$ is coprime to $30$, so, in particular, $c_i(2^{n+i}-1)/(n+i)$ is a rational number whose denominator is divisible by a prime $p\ge 7$. Here, $c_i=6,60,30$, for $i=1,3,5$, respectively. Clearly, $p$ cannot divide $n+j$ for $j\ne i$, $j\in \{1,2,3,4,5,6\}$, therefore $S(n,6)$ is a rational number whose denominator is a multiple of $p$.    

\medskip

So far, we reproved the main result from \cite{3} and even proved the cases $r=5$ and $r=6$. In order to extend our argument to cover all $r$, we need two ingredients:
\begin{itemize}
\item[(i)] A general formula of the shape of (1.2) valid for $n$ and $r$;
\item[(ii)] A statement about prime factors of consecutive integers, namely that under some mild hypothesis, out of every $r$ consecutive integers there is one of them divisible by a prime larger than $r$.
\end{itemize}

The next statement takes care of (i) and, in particular, justifies formulas \eqref{eq:1to5}.

\begin{lemma}
We have
\begin{equation}
\label{eq:formula}
S(n,r)=\sum_{j=0}^{r-1} (-1)^{r-1-j} r \binom{r-1}{j} \left(\frac{2^{n+j+1}-1}{n+j+1}\right).
\end{equation}
\end{lemma}

\begin{proof}
\begin{eqnarray*}
S(n,r) & = & r\sum_{k=0}^n \binom{n}{k} \frac{1}{k+r}=  r\sum_{k=0}^n \binom{n}{k} \int_0^1 x^{k+r-1} dx\\
& = & r\int_0^1 \left(\sum_{k=0}^n \binom{n}{k} x^{k+r-1}\right) dx=  r\int_0^1 \left(\sum_{k=0}^n \binom{n}{k} x^k\right) x^{r-1} dx\\
& = & r\int_0^1 (1+x)^n x^{r-1} dx  =  r\int_0^1 (1+x)^n (1+x-1)^{r-1} dx \\
& = &  r\int_0^1 (1+x)^n \left(\sum_{j=0}^{r-1} (-1)^{r-1-j} \binom{r-1}{j} (1+x)^j \right) dx\\
& = & \int_0^1 \left(\sum_{j=0}^{r-1} (-1)^{r-1-j} r\binom{r-1}{j}  (1+x)^{n+j}\right) dx  \\
&= & \sum_{j=0}^{r-1} (-1)^{r-1-j} r\binom{r-1}{j} \int_{0}^1 (1+x)^{n+j} dx\\
& = & \sum_{j=0}^{r-1} (-1)^{r-1-j} r\binom{r-1}{j} \left(\frac{2^{n+j+1}-1}{n+j+1}\right).
\end{eqnarray*}
\end{proof}

For (ii), let us recall Sylvester's extension of Bertrand's postulate (see \cite{2}).

\begin{theorem}
If $n \ge r\ge 2$, then one of the numbers $n+1,~n+2,~\ldots,~n+r$ is divisible by a prime larger than $r$.
\end{theorem}

However, Sylvester's theorem is not enough to prove that $S(n,r)$ is not an integer for any $n$ and $r$, even when $n\ge r$, because although we infer that there exists 
$i\in \{1,2,\ldots,r\}$ such that $p\mid n+i$ for some prime $p>r$, and $n+i$ does not divide $2^{n+i}-1$, it is still possible that 
$c_i(2^{n+i}-1)/(n+i)$ is a rational number whose denominator is {\bf not} divisible by $p$, and therefore we cannot infer that $p$ divides the denominator of $S(n,r)$. 
However, Sylvester's theorem is enough to deal with the case $n\le r-1$. Namely, in this case, we work directly with the original representation of \eqref{eq:Snr}, which is
$$
S(n,r)=1+\sum_{j=1}^n \frac{r}{r+j} \binom{n}{j}.
$$
If $r+1>n$, then, again by Sylvester's theorem, one of the numbers $r+1,r+2,\ldots,r+n$ is divisible by a prime $p>n$. Such a prime does not divide
$\binom{n}{j}$ for any $j\in \{1,\ldots,n\}$, and does not divide $r$ either (otherwise, it divides both $r$ and $r+j$ for some $j\in \{1,\ldots,n\}$, so it divides their difference, which is a number $\le n$, a contradiction).  So, it remains to deal with $r=n+1$. In this case, we apply Bertrand's postulate, to conclude that there is a prime $p\in ((n+1),2n+1]$. This prime divides neither $n+1$ nor $\binom{n}{j}$ for $j\in \{1,\ldots,n\}$, so $p$ divides the denominator of $S(n,n+1)$. 

To summarize, in this note we proved, in addition to formula \eqref{eq:formula}, the following partial results towards the conjecture that $S(n,r)$ is not an integer for any positive integers $n$ and $r$:

\begin{theorem}
\
\begin{itemize}
\item[(i)] $S(n,r)$ is not an integer for any $r\in \{1,2,3,4,5,6\}$ and $n\ge 2$;
\item[(ii)] $S(n,r)$ is not an integer for $1\le n\le r-1$. 
\end{itemize}
\end{theorem}

\end{document}